\documentclass[10pt,reqno]{amsart}

\usepackage{amscd}
\usepackage{amsthm}
\usepackage{amssymb}

\usepackage{enumerate, array,mathrsfs, a4wide}
\usepackage{graphicx}
\usepackage{xcolor}

\usepackage{rotating}
\usepackage{bbm}
\usepackage[all,cmtip]{xy}

\usepackage[utf8]{inputenc}
\usepackage{dsfont}

\usepackage{hyperref}
\usepackage{pseudocode}

\usepackage[
  backend=biber,
  style=alphabetic,
  useprefix=true, 
  doi=false,
  url=false,
  isbn=false,
  date=year,
  clearlang=true,
  maxbibnames=99,
]{biblatex}
\newtheorem{theorem}{Theorem}[section]

\newtheorem{lemma}[theorem]{Lemma}
\newtheorem{proposition}[theorem]{Proposition}

\newtheorem*{theorem*}{Theorem}

\theoremstyle{definition}
\newtheorem{definition}[theorem]{Definition}

\theoremstyle{remark}
\newtheorem{remark}[theorem]{Remark}

\numberwithin{equation}{section}

\newcommand{\C}{{\mathbb{C}}}
\newcommand{\R}{{\mathbb{R}}}

\newcommand{\Z}{{\mathbb{Z}}}

\renewcommand{\epsilon}{\varepsilon}
\renewcommand{\phi}{\varphi}
\renewcommand{\theta}{\vartheta}

\newcommand{\w}{\wedge}

\DeclareMathOperator{\id}{Id}
\DeclareMathOperator{\im}{Im}

\usepackage[foot]{amsaddr}

\begin{document}
\title[Computational symplectic topology and RTBP]{Computational symplectic topology and symmetric orbits in the restricted three-body problem}

\author{Chankyu Joung}
\author{Otto van Koert$^{\ast}$}

\address{Department of Mathematical Sciences and Research Institute of Mathematics, Seoul National University\\
Building 27, room 439 and 402\\
San 56-1, Sillim-dong, Gwanak-gu, Seoul, South Korea\\
Postal code 08826 }
\address{$^{\ast}$Corresponding author: Otto van Koert, }
\email{okoert@snu.ac.kr}



\begin{abstract}
In this paper, we propose a computational approach to proving the Birkhoff conjecture on the restricted three-body problem, which asserts the existence of a disk-like global surface of section. Birkhoff had conjectured this surface of section as a tool to prove existence of a direct periodic orbit.
Using techniques from validated numerics we prove the existence of an approximately circular direct orbit for a wide range of mass parameters and Jacobi energies. 
We also provide methods to rigorously compute the Conley-Zehnder index of periodic Hamiltonian orbits using computational tools, thus giving some initial steps for developing computational Floer homology and providing the means to prove the Birkhoff conjecture via symplectic topology.
We apply this method to various symmetric orbits in the restricted three-body problem.
\end{abstract}

\maketitle


\section{Introduction}
In his 1915 paper on the restricted three-body problem, Birkhoff proved the existence of the retrograde, periodic orbit in the bounded component of the restricted three-body problem.
He also conjectured the existence of a disk-like global surface of section, and hoped that this could be useful to prove the existence of a direct periodic orbit.
For small mass parameter $\mu$, such a disk-like surface of section was established by McGehee using perturbative means, \cite{mcgehee_phd_thesis_1969}. The binding orbit of this global surface of section is the so-called retrograde orbit found by Birkhoff. For small mass parameter, this orbit is unique and non-degenerate, a fact that was already known to Poincar\'e.
In the non-perturbative case, far less is known. The existence of a disk-like global surface of section has not yet been proved for more general mass parameters. 
If it exists, a potential binding orbit of this global surface of section would be the retrograde orbit, but even basic properties such as non-degeneracy and uniqueness of that orbit are unknown.

This brings us to symplectic geometry techniques.
In the last ten years, various results have been established that allow the usage of holomorphic curve methods from symplectic geometry for this classical problem. 
In particular, a result of Hofer, Wysocki and Zehnder, \cite{hofer_wysocki_zehnder_1998}, states that dynamically convex energy hypersurfaces in $\R^4$ admit disk-like global surfaces of section.
Hence one would hope that dynamical convexity can be proved and settle the Birkhoff conjecture this way.
Despite some progress in this direction, dynamical convexity has remained out of reach, but numerical experiments suggest that the retrograde orbit is the orbit with minimal action and minimal Conley-Zehnder index, namely Conley-Zehnder index equal to $3$. Proving this would be a step forward for the Birkhoff conjecture.

In this paper, we combine methods of validated numerics, \cite{tucker_validated_numerics_2011,kapela_capd_2021}, with techniques from symplectic dynamics in order to make progress on this problem.
Specifically, we shall prove the existence of a direct periodic orbit and establish several basic properties. We also study symmetric periodic orbits of small action and show that the retrograde orbit has minimal action. This is done using a covering argument. With additional computational effort, such an argument can also be used to prove dynamical convexity.

In order to state our results more precisely, let us define the restricted three-body problem.
First of all, we fix a parameter $\mu \in [0,1]$, called the {\bf mass parameter}. 
The phase space is $\R^2 \setminus \{ (-\mu,0), ~(1-\mu,0) \} \times \R^2$, which we identify with $\C \setminus \{ -\mu,~1-\mu \} \times \C$.
In our conventions, the Jacobi Hamiltonian of the planar, restricted $3$-body problem is then given by
\begin{equation}
\label{eq:RTBP_rot_unreg}
H=\frac{1}{2}|p|^2+q_1p_2-q_2p_1-\frac{1-\mu}{|q+\mu|}-\frac{\mu}{|q-1+\mu|}.
\end{equation}
Define $\Sigma_c:=H^{-1}(-c)$. The term $L:=q_1p_2-q_2p_1$ is called the angular momentum. 
We call $c$ the {\bf energy level}. 

The astronomy literature, for example page 49 from \cite{murray_dermott_solar_1999}, defines a retrograde spatial orbit as an orbit with inclination greater than 90 degrees, and spatial prograde orbit as an orbit with inclination less than 90 degrees. For the case of planar orbits, which we consider, this means that retrograde orbits move in the opposite direction to the motion of the primaries. A precise definition and characterization in terms of angular momentum is given in Section~\ref{sec:prelims}.

The flow of the Jacobi Hamiltonian is not complete, but it can be regularized. This is still helpful even though we are not interested in collision orbits, because
\begin{itemize}
    \item the regularized system has a better controlled numerical error;
    \item the search for periodic orbits does not get stuck in collision orbits, allowing us to obtain more complete results. 
\end{itemize}
We will describe Levi-Civita and Moser regularization in Section~\ref{sec:LC-regularization}.

For small mass parameter $\mu$, one can prove the existence of the direct orbit using integrability of the rotating Kepler problem as was already pointed out by Birkhoff, see page 327 of \cite{birkhoff_restricted_1915}. With a computer-assisted proof, we can also prove the existence in the non-perturbative case.
\begin{theorem}
\label{thm:direct_existence}
For all $(\mu,c)$ with $\mu \in[0.1,0.5]$ and Jacobi energy $c$ between and including the first critical energy and $2.1$, there is a symmetric, periodic orbit that is direct and crosses the $q_1$-axis exactly twice.
This orbit is non-degenerate.
\end{theorem}
The specific choice of Jacobi energy $2.1$ as largest energy parameter is somewhat arbitrary, and was chosen since computer-assisted techniques require, roughly speaking, a compact set of parameters and this parameter is larger than the critical energy for all $\mu$.
Conley has proved the existence of direct and retrograde periodic orbits for $c \gg 2$ in \cite{conley_long_1963}.
This makes it plausible that there is a direct periodic orbit for all Jacobi energies below the critical value. However, Conley's proof is perturbative, so we cannot draw this conclusion without additional work.

{
To state our result concerning the Birkhoff conjecture and dynamical convexity, we define the action of a periodic Hamiltonian orbit by
$$
\mathcal A_{\mu,c}(\gamma)=\int_{\gamma} \lambda,
$$
where $\lambda$ is the Liouville form, which we choose as $\lambda=-q\cdot dp=p\cdot dq-d(p\cdot q)$. We include the parameters $\mu$ and $c$ because the regularized Hamiltonian explicitly depends on these.
The Conley-Zehnder index of a periodic Hamiltonian orbit will defined in Section~\ref{sec:degree_def_CZ}; roughly speaking, it can be thought of as a winding number of the linearized flow and plays an important role in Floer theory, see Remark~\ref{rem:bifurcation_floer}.
Together with the self-linking number, these invariants play a role in proving the Birkhoff conjecture \cite{hryniewicz_salomao_tight_3_sphere_2011,hryniewicz_fast_2012,hryniewicz_salomao_wysocki_genes_zero_2023}.
}
We have the following partial result.
\begin{theorem}
\label{thm:small_action_unknotted_and_sl}
For every $\mu \in [0, 0.5]$ and $c\in [2.1, 2.1 +10^{-6}]$, there are $A_{\mu,c}>0$, a retrograde symmetric periodic orbit $\gamma^r_{\mu,c}$ and a direct symmetric periodic orbit $\gamma^d_{\mu,c}$ such that 
\begin{enumerate}
    \item symmetric periodic orbits with action $\mathcal A_{\mu,c}$ less than $A_{\mu,c}$ are reparametrizations of $\gamma^r_{\mu,c}$ or $\gamma^d_{\mu,c}$; 
    \item the orbits $\gamma^r_{\mu,c}$ and $\gamma^d_{\mu,c}$ are non-degenerate, unknotted and have self-linking number $-1$ in the Levi-Civita regularization;
    \item in the Moser regularization the orbits $\gamma^r_{\mu,c}$ and $\gamma^d_{\mu,c}$ have Conley-Zehnder index $1$ and $3$, respectively;
    \item in the Levi-Civita regularization the orbits $\gamma^r_{\mu,c}$ and $\gamma^d_{\mu,c}$ have Conley-Zehnder index $3$ and $5$, respectively.
\end{enumerate}
\end{theorem}

This theorem is proved by covering the fixed point locus of an anti-symplectic involution with sufficiently small sets for which either existence or non-existence of a short-action periodic orbit can be shown using computer-assisted techniques. The non-degeneracy and index computation can be established afterwards by investigating the linearized flow.
In principle, this method could be used to cover the entire connected component of the energy hypersurface to check for all periodic orbits up to a given action bound. Together with a bound on the growth rate of the index, which can also be obtained from a covering argument, this would be enough to check dynamical convexity. 

In view of the following result of Hryniewicz, see Theorem~1.5 from \cite{hryniewicz_fast_2012}, we expect the retrograde and direct orbits to be binding orbits for a disk-like surface of section in this range of parameters.
\begin{theorem}[Hryniewicz]
\label{thm:Hryniewicz_gss}
Consider a dynamically convex, starshaped hypersurface $\Sigma \subset \mathbb{C}^2$ and suppose that $\gamma$ is an unknotted periodic Reeb orbit with self-linking number equal to $s\ell(\gamma)=-1$. 
Then $\gamma$ bounds a disk-like global surface of section.
\end{theorem}

\begin{remark}
    Upcoming work of Bowen Liu and the second author, \cite{liu_van_Koert_convexity_inpreparation}, will prove convexity of the Levi-Civita embedding for a wide range of parameters, including $\mu=1/2$ and $c=2.1$. In combination with the result of Hryniewicz, this shows that the retrograde and direct orbits from Theorem~\ref{thm:small_action_unknotted_and_sl} both bind a disk-like global surface of section. In order to give a complete application of computational symplectic topology to this problem, we include an appendix, Section~\ref{sec:appendix}, where we prove relevant convexity results.
\end{remark}

\begin{theorem}
\label{thm:gss}
    For every $\mu \in [0, 0.5]$ and $c\in [2.1, 2.1+10^{-6}]$, the compact component of the Levi-Civita regularization is convex. In particular, the retrograde and direct orbits from Theorem~\ref{thm:small_action_unknotted_and_sl} both bind a disk-like global surface of section for these parameter values.
\end{theorem}

The structure of this paper is as follows.
In Section~\ref{sec:prelims} we give background and give a definition of retrograde and direct orbits that is compatible with what is commonly used in astronomy. We include a brief discussion on regularizations, review the Conley-Zehnder index and develop some new formulas to compute this index, which are computationally robust and can be used in validated numerics.
In Section~\ref{sec:proofs} we state some additional results on properties of the retrograde and direct orbits and prove the theorems. A proof of the convexity properties can be found in the appendix.
The computational part is separate and can be found on \url{https://github.com/ckjoung/soir}.

\section{Preliminaries}
\label{sec:prelims}
We start by recalling how to derive the Jacobi Hamiltonian since we will need the transformations from this derivation below. 
In fixed (or sidereal) coordinates $(\bar q, \bar p)$, the time-dependent Hamiltonian for the restricted three-body problem is given by
$$
H_s=\frac{1}{2}|\bar p|^2-\frac{1-\mu}{|\bar q-\bar e(t)|}-\frac{\mu}{|\bar q-\bar m(t)|},
$$
where the heavy primaries $\bar e(t)$, ``Earth'', and $\bar m(t)$, ``Moon'', are given by
$$
\bar e(t)=\left(
\begin{array}{c}
-\mu \cos(t) \\
\mu \sin(t)
\end{array}
\right)
\quad
\bar m(t)=\left(
\begin{array}{c}
(1-\mu) \cos(t) \\
-(1-\mu) \sin(t)
\end{array}
\right)
$$
In other words, the two heavy primaries move around each other in a clock-wise motion in our conventions.
We will shorten these formulas by defining the time-dependent rotation $R_t$,
$$
R_t
=
\left(
\begin{array}{cc}
\cos(t) & -\sin(t) \\
\sin(t) & \cos(t)
\end{array}
\right)
.
$$
So $\bar e(t)=R_{-t} (-\mu,0)$ and $\bar m(t)=R_{-t}(1-\mu,0)$. The sidereal coordinates and rotating coordinates are related by the formula
$$
\bar q(t)=R_{-t} q(t),\quad \bar p(t)=R_{-t} p(t).
$$
By noting that the symplectic transformation of $T^*\R^2$ defined by $\psi(t)=R_t\oplus R_t$ is generated by $q_1 p_2 -q_2p_1$ and applying the composition formula for Hamiltonians, we see that the time $t$-flow of $H$, which we denote by $Fl^{X_{H}}_t$, satisfies
$$
Fl^{X_{H}}_t = \psi(t) \circ Fl^{X_{H_s}}_t.
$$
This means that the Hamiltonian~\eqref{eq:RTBP_rot_unreg} generates the flow of $H_s$ but in rotating coordinates.\footnote{
With the standard symplectic form $\omega =dp \w dq$, we define the Hamiltonian vector field of $H$ via the equation $i_{X_H}\omega =-dH$; the flow equations are then given by $\dot q =\frac{\partial H}{\partial p},~\dot p =-\frac{\partial H}{\partial q}$.
}

Following standard conventions, we define (spatial) retrograde orbits as orbits whose inclination is greater than 90 degrees, and prograde orbits as orbits with an inclination of less than 90 degrees.
Since we will be considering planar orbits, this means intuitively that retrograde orbits move in the opposite direction to the motion of the primaries.
To be more precise, we give the following definition.
{
\begin{definition}
We call a path $\bar q(t)$ in $\R^2$ with sidereal coordinates {\bf retrograde} with respect to $\bar e(t)$ if the angle
$$
\phi:=\arg
\left(
\bar q_1(t) - \bar e_1(t) +i(\bar q_2(t) - \bar e_2(t) )
\right)
$$
is increasing and {\bf prograde} (or {\bf direct}) if this angle is decreasing. If $q(t)$ is a path in rotating coordinates, then we call it retrograde or direct with respect to $e(t)$ if the underlying path in sidereal coordinates is retrograde or direct, respectively.
\end{definition}
}

This definition depends on our conventions that the primaries move around each other in a clock-wise fashion.
With a computation we get the more easily checked condition in rotating coordinates,
\begin{proposition}
Suppose that $(q(t),p(t) )$ is an orbit of $H$ in rotating coordinates. Then the curve $q(t)$ is retrograde with respect to $e(t)$ if
\begin{equation}
\label{eq:retro_condition}
    (q_1+\mu)p_2 -q_2p_1 > \mu (q_1+\mu).
\end{equation}
The orbit is direct if $(q_1+\mu)p_2 -q_2p_1 < \mu (q_1+\mu)$.
\end{proposition}
We note that the left-hand side is the angular momentum with respect to $(-\mu,0)=\bar e(0)$, which we can write as
$$
L_\mu = (q+e(0)) J p=
\left(
\begin{array}{c}
q_1+\mu \\
q_2
\end{array} 
\right)^t
\left(
\begin{array}{cc}
0 & 1 \\
-1 & 0
\end{array} 
\right)
\left(
\begin{array}{c}
p_1 \\
p_2
\end{array} 
\right)
$$

\begin{proof}
We have $\bar q(t)=R_{-t}q(t)$, and $\bar e(t)=R_{-t} (-\mu,0)$, so taking the derivative of $\phi$, we find using the equations of motion
\[
\begin{split}
\dot \phi &= \frac{(\bar q-\bar e)^t J \frac{d}{dt} (\bar q -\bar e)}{\Vert \bar q -\bar e \Vert^2}
= \frac{( q- e)^t R_{-t}^t J \frac{d}{dt} R_{-t}(q -e)}{\Vert q -e \Vert^2} \\
&=-1 + \frac{( q- e)^t J \frac{d}{dt}q}{\Vert q -e \Vert^2}
=-1 +\frac{(q_1+\mu)p_2 +q_1^2 +\mu q_1 -q_2 p_1 +q_2^2}{\Vert q -e \Vert^2} \\
&=\frac{L_\mu-\mu q_1 -\mu^2}{\Vert q -e \Vert^2},
\end{split}
\]
from which the conclusion follows.
\end{proof}

To indicate what this definition means, we plot in Figure~\ref{fig:retrograde} the symmetric periodic orbits for $\mu=0.9$ found by the Birkhoff shooting procedure, which are retrograde in the sense of Birkhoff. We also plot the quantity $L_\mu - \mu q_1 - \mu^2$ as a function of time, whose sign indicates whether the angle $\phi$ is increasing or decreasing.
We see from the plot that the periodic orbit at Jacobi energies including $c=1.2$ and $c=1.0$ are not retrograde with respect to the primary on the left (the light primary), despite going around it in counter clock-wise fashion. Put simply, the movement of the primary is too fast at some times, making the apparent movement as seen from the light primary prograde.

\begin{figure}[!t]
    \centering
    \includegraphics[width=0.8\textwidth]{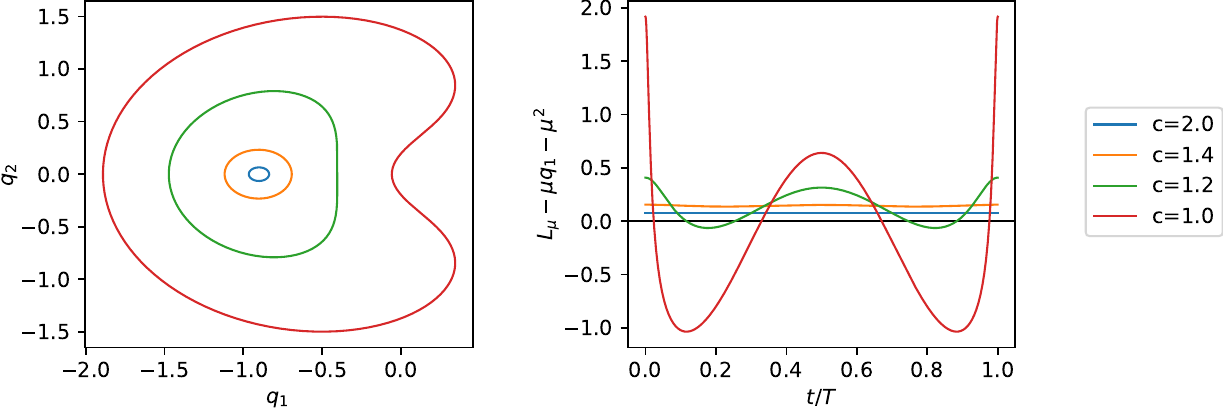}
    \caption{Retrograde orbits and orbits that are not apparently retrograde for $\mu=0.9$}
    \label{fig:retrograde}
\end{figure}

\begin{remark}
    In his paper \cite{birkhoff_restricted_1915}, Birkhoff does not give a definition of a retrograde orbit, but uses the term to refer to ``upward'' going orbits on the positive real axis, i.e.~orbits whose initial points are of the form $(q_1,0;0,p_2)$ with $p_2>0$. The definition we give here is a priori a stronger requirement, but for Jacobi energies considered in this paper, the Birkhoff orbit is also retrograde in the astronomical sense.
\end{remark}

\subsection{Hill's region}
We consider RTBP with mass parameter $\mu \in (0,1)$.
The corresponding Hamiltonian $H$ has five critical points, which correspond to the five critical points of the effective potential, given by
$$
U=-\frac{1}{2}|q|^2 -\frac{1-\mu}{|q+\mu|}
-\frac{\mu}{|q+\mu-1|}.
$$
The five critical points of $U$ are called Lagrange points and we write them as $\ell_1,\ldots,\ell_5$ and the corresponding critical points of $H$ by $L_1,\ldots,L_5$.
All these critical points are non-degenerate, where $L_1,\ldots, L_3$ are critical points of Morse index $1$, while the remaining two critical points have index $2$. 
Unless indicated otherwise, we will assume throughout this paper that $-c \leq H(L_1)$, where $L_1$ is the critical point with the smallest energy value.

The topology of the levels of $U$ (and accordingly of $H$) changes as $-c$ passes critical values. A sketch of these level sets is shown in Figure~\ref{fig:hillsregions}.
The closure of the sublevel sets of $U$ are called Hill's regions.
It is well-known, see \cite[Section 7]{birkhoff_restricted_1915} or \cite[Chapter 5.5]{frauenfelder_restricted_2018}, that the Hill's region has three connected components for $-c<H(L_1)$. The same holds true for the energy hypersurface $H^{-1}(-c)$.
None of these components are compact, and we will consider the connected component of $H^{-1}(-c)$, for which the closure of the projection to the $q$-coordinates contains the point $(-\mu,0)=-\mu \in \mathbb{C}$.
We will give this component a name after compactifying it in the next section.

\begin{figure}[!t]
    \centering
    \includegraphics[width=0.5\textwidth]{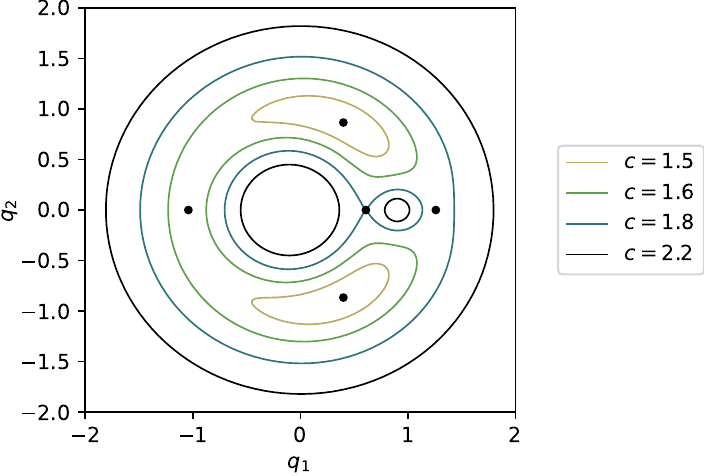}
    \caption{Plot of the boundaries of the Hill's regions for $\mu=0.1$ for various energy levels, with the Lagrange points indicated as dots. The paper is mostly concerned with orbits lying in the center region.}
 \label{fig:hillsregions}
\end{figure}

\subsection{Levi-Civita and Moser regularization}
\label{sec:LC-regularization}
To deal with singularities, we will apply the Levi-Civita regularization at $-\mu$. 
The Levi-Civita regularization is defined by the coordinate transformation
$$
q+\mu=2z^2,
\quad
p =\frac{w}{\bar z}.
$$
First note that the Hamiltonian for RTBP is given in complex coordinates by
$$
H=\frac{1}{2}|p|^2+\im(p\bar q)
-\frac{1-\mu}{|q+\mu|}
-\frac{\mu}{|q+\mu-1|}.
$$
After the above coordinate transformation, we get
$$
\tilde H=\frac{1}{2}|\frac{w}{\bar z}|^2
+2(z_1w_2-z_2w_1)
-\mu(z_1w_2+z_2w_1)/|z|^2
-\frac{1-\mu}{|2z^2|}-\frac{\mu}{|2z^2-1|}.
$$
The singularity at the origin can now be resolved. We define the Levi-Civita Hamiltonian by shifting and rescaling $\tilde H$, giving us
\begin{equation}
\label{eq:LC_hamiltonian}
\begin{split}
K_{\mu,c} &:=(\tilde H+c)\cdot |z^2| \\
&=\frac{1}{2}|w|^2+c|z|^2 -\frac{1-\mu}{2} +2|z|^2(z_1w_2-z_2w_1)
-\mu(z_1w_2+z_2w_1)
-\frac{\mu |z|^2}{|2z^2-1|}.
\end{split}
\end{equation}
On the level set $K_{\mu,c}=0$, the vector field $X_{K_{\mu,c}}$ is a multiple of the Hamiltonian vector field of $H$ on the level set $H=-c$,
$$
X_{K_{\mu,c}}=|z|^2 X_{H+c} +(H+c)\cdot X_{|z|^2},
$$
so the dynamics of $X_{K_{\mu,c}}$ are a reparametrization of the dynamics of the Hamiltonian $H$ at energy $-c$.
We will denote the component of the energy surface $K_{\mu,c}=0$ corresponding to the primary at $q=(-\mu,0)$ by $\Sigma_{\mu,c}$.  

\subsubsection{Moser regularization}
Due to the quadratic map of the Levi-Civita regularization, each point in an unregularized energy hypersurface lifts to two points. Moser regularization provides a scheme which results in a unique lift by using a map into the cotangent bundle of a sphere.
We write the unit sphere as
$$
S^n = \{ (\xi_0,\vec \xi) \in \R^{n+1}~|~ \xi_0^2 +|\vec \xi|^2=1 \}
$$
and then proceed as follows. We switch the role of position and momentum coordinates by setting 
$$
\vec x = \vec p \quad\text{and}\quad \vec y =-\vec q,
$$
and then consider the map
$$T^*\R^2 \longrightarrow T^*S^2 \subset T^*\R^3, \quad (\vec x, \vec y) \longmapsto (\xi_0, \vec \xi; \eta_0, \vec \eta)$$
where
$$
\xi_0 = \frac{|\vec x|^2 - 1}{|\vec x|^2 + 1},
\quad
\vec \xi = 2\frac{\vec x}{|\vec x|^2 + 1},
\quad
\eta_0 = \vec x \cdot \vec y,
\quad
\vec \eta =
\frac{|\vec x|^2+1}{2} \vec y -(\vec x \cdot \vec y) \vec x.
$$
This map preserves the canonical $1$-form by construction, so it satisfies the relation
$$
\vec y \cdot d\vec x = \eta_0 d\xi_0 +\vec \eta \cdot d\vec \xi.
$$
After rescaling the time-parametrization, the Hamiltonian describing the Moser regularized RTBP is given by
$$
Q_{\mu,c}
=\frac{1}{2}f(\xi,\eta)^2 |\eta|^2 
$$
with
\begin{equation}\label{eq:f3bp}
f(\xi,\eta)=
1+ \left( 1-\xi_{{0}} \right)  \left( -(c+1/2)+ \xi_{{2}}\eta_{{1}}-\xi_{{1}}
\eta_{{2}} \right) -\xi_{{2}} \left( 1-\mu \right)-\frac {(1-\mu) (1-\xi_{{0}}) }{\Vert \vec \eta (1-\xi_0)+\vec \xi \eta_0 + \vec m - \vec e \Vert}
\end{equation}
By \cite{albers_contact_2012}, every level set of this Hamiltonian below the first critical value has a component which is fiberwise starshaped, so the vector field $\eta\cdot \frac{\partial}{\partial \eta}$ is transverse to this component. With this in hand, we obtain the following well-known result.
\begin{proposition}
\label{prop:starshaped}
If $-c$ is less than the smallest critical value of the Hamiltonian $H$, then $\Sigma_{\mu,c}$ is starshaped, so in particular diffeomorphic to $S^3$. Furthermore, this component forms a $2-1$ cover of the corresponding component in the Moser regularization, which is therefore diffeomorphic to $\R P^3$.
\end{proposition}
To see this, we note that both the Moser regularization and the Levi-Civita regularization preserve the canonical $1$-form. This implies that $\Sigma_{\mu,c}$ is starshaped, see Lemma~4.2.1 from \cite{frauenfelder_restricted_2018}. It follows that $\Sigma_{\mu,c}$ is diffeomorphic to $S^3$.
We note that $\Sigma_{\mu,c}$ is invariant under the symplectic involution 
\begin{equation}
    \label{eq:sympl_involution}
    s:(z,w) \longmapsto (-z,-w).
\end{equation}
as can be seen from the definition of the Levi-Civita transformation, and this implies the last claim.

\subsubsection{Contact type and action}
A hypersurface $\Sigma$ in a symplectic manifold $(W,\omega)$ is of contact type if there is a vector field $X$ defined in a neighborhood of $\Sigma$ such that 
\begin{itemize}
    \item $X$ is transverse to $\Sigma$
    \item $X$ is a Liouville vector field, so $d(i_X \omega)=\omega$.
\end{itemize}
The contact form is then defined as $\alpha = (i_X \omega)|_{\Sigma}$, i.e.~it is the restriction of the Liouville form $\lambda =i_X \omega$. The above remarks show that $\Sigma_{\mu,c}$ is of contact type. 
The Reeb vector field $R$ of a contact form $\alpha$ is defined by $i_R d\alpha=0$ and $\alpha(R)=1$.
If $\Sigma$ is, in addition, a level set of an autonomous Hamiltonian $H$, then $X_H$ is a multiple of the Reeb field on $\Sigma$.

The {\bf action} of an orbit $\gamma$ in a contact manifold $(Y,\alpha)$ is given by
$$
\mathcal A(\gamma) =\int_{\gamma} \alpha.
$$
As $\alpha(R)=1$, the action of a Reeb orbit is strictly increasing with time.
If we restrict $\mathcal A$ to periodic orbits, then one can show with variational calculus that critical points of this functional are unparametrized Reeb orbits. These are then also Hamiltonian orbits if $\Sigma=H^{-1}(0)$.

\begin{remark}
The action is independent of the parametrization, so if $\Sigma$ is the zeroset of a Hamiltonian $H$ that is of contact type, then the action of a Hamiltonian orbit $x$ can also be computed as
$$
\mathcal A(x) = \int_0^T \alpha(\dot x(t))dt=\int_{0}^T \lambda(X_H(x(t))) dt.
$$
\end{remark}

\subsection{Frames and local invariants of orbits}
\label{sec:frames_invariants}
In general, one chooses Seifert surfaces to define local frames along orbits, but in our setting we have a convenient global frame.
To see this, consider an autonomous Hamiltonian $H$ on the phase space $\R^4=\{(q_1,q_2,p_1,p_2)\}$ with symplectic form $\omega = dp\wedge dq$.
{
Since our $q$-coordinates come first, this means that the matrix representation of $\omega$ is given by
$$
\Omega_4:=
[\omega]=
\left(
\begin{array}{cccc}
0 & 0 & -1 & 0 \\
0 & 0 & 0 & -1 \\
1 & 0 & 0 & 0 \\
0 & 1 & 0 & 0
\end{array}
\right)
.
$$
}
On a regular level set $\Sigma$ of $H$, {we have $\omega(\nabla H,X_H) = -\omega(X_H,\nabla H) = dH( \nabla H)=\Vert \nabla H \Vert^2>0$}, so the two vectors 
$$
Z:=\frac{1}{\Vert \nabla H \Vert^2}\nabla H, \quad X_H=I\nabla H
$$ 
form a symplectic frame{, i.e.~$\omega(Z,X_H)=1$}. These span a rank-2 symplectic vector bundle $L$.
We take the symplectic complement as a convenient way to choose a transverse slice.
In our setting, this can be done explicitly by means of the quaternionic matrices
$$
I = 
\left(
    \begin{array}{cccc}
     0 &  0 &  1 &  0\\
     0 &  0 &  0 &  1\\
    -1 &  0 &  0 &  0\\
     0 & -1 &  0 &  0
    \end{array}
\right),\quad 
J = 
\left(
    \begin{array}{cccc}
     0 &  1 &  0 &  0\\
    -1 &  0 &  0 &  0\\
     0 &  0 &  0 & -1\\
     0 &  0 &  1 &  0
    \end{array}
\right),\quad 
K = 
\left(
    \begin{array}{cccc}
     0 &  0 &  0 & -1\\
     0 &  0 &  1 &  0\\
     0 & -1 &  0 &  0\\
     1 &  0 &  0 &  0
    \end{array}
\right).
$$
Here, the vectors $X_H = I\nabla H, J \nabla H, K \nabla H$ trivialize $T\Sigma$. We obtain a symplectic frame of the symplectic complement $L^\omega$, {which is defined as
$$
L^\omega=\{ (x,v) \in T\R^4 = \R^4\times \R^4~|~x\in \Sigma,~\omega(v,w)=0 \text{ for all } w\in L_x=\R Z(x) \oplus \R X_H(x) \}
$$
}
by setting
\begin{equation}
\label{eq:frame_contact_structure}
U:=\frac{1}{\Vert \nabla H \Vert} J \nabla H,\quad V:=\frac{1}{\Vert \nabla H \Vert} K \nabla H.
\end{equation}
{In particular, we have $\omega(U,V)=1$.}
We will now assume that $\Sigma$ is star-shaped, so $\Sigma$ is diffeomorphic to $S^3$.
We can use the first vector of the above frame to define the self-linking number.
\begin{definition}
    Suppose that $\gamma$ is a periodic Hamiltonian orbit. Define the pushoff of $\gamma$ by flowing slightly in direction of $U$, say $\tilde \gamma=Fl^U_\epsilon \circ \gamma$.
    Then the {\bf self-linking number} of $\gamma$ is defined as the linking number of $\gamma$ and $\tilde \gamma$, so $sl(\gamma)=lk(\gamma,\tilde \gamma)$.
\end{definition}
A more general definition of the self-linking number is given in Definition~3.5.28 of \cite{geiges_introduction_2008}.

\begin{definition}
    We say a periodic Hamiltonian orbit of period $T_0$ is {\bf (transversely) non-degenerate} if $\ker(dFl^{X_H}_{T_0} -\id)|_{T\Sigma}=\R X_H$.
\end{definition}
In practice, we choose a complement of $X_H$ in $T\Sigma$ and show that the restriction to this complement has no eigenvalues equal to $1$. In case $\Sigma$ is $3$-dimensional, as is the case for us, we can use Lemma~\ref{lemma:transverse_nondegeneracy}.
For non-degenerate orbits, we will define an invariant, known as Conley-Zehnder index. 
In order to define it, we need a symplectic trivialization, which we construct with the above frame
\[
\epsilon: \Sigma_{\mu,c} \times \R^2 \longrightarrow L^\omega,
\quad (x;u,v) \longmapsto u\cdot U(x) +v \cdot V(x).
\]
We use this trivialization to obtain a path of symplectic matrices representing the transverse linearized flow
\begin{equation}
\label{eq:transverse_linearized_path}
\Psi(t):=
\epsilon(Fl^{X_H}_t(x_0), {}\cdot{} )^{-1} 
\circ
pr_{L^\omega}
\circ
d_{x_0}Fl^{X_H}_t
\circ
\epsilon(x_0, {}\cdot{} ),
\end{equation}
{where $pr_{L^\omega}$ is the projection $T\R^4 \to L^\omega$, given by
$$
pr_{L^\omega}(a Z+b X_{H} +c U + dV)= cU +dV.
$$
\begin{remark}
The formula~\eqref{eq:transverse_linearized_path} can be simplified by using the following observations.
By \cite{albers_contact_2012} we know that $\Sigma_{\mu,c}$ is of contact type with contact form $\alpha= \lambda|_{\Sigma_{\mu,c}}$, and the associated Reeb flow is a reparametrization of the flow of $X_H$. For a Reeb flow, one can choose a frame $\bar U, \bar V$ trivializing $\ker \alpha$, which is invariant under the Reeb flow. For such a frame, the projection in \eqref{eq:transverse_linearized_path} is not needed.
\end{remark}

We will now give a matrix representation of formula~\eqref{eq:transverse_linearized_path}.
Given the symplectic $U,V$-frame, so $\omega(U,V)=1$, we obtain a matrix representation $\Omega_2$ for the symplectic form on $L^\omega$. This is
$$
\Omega_2=
\left(
\begin{array}{cc}
0 & 1 \\
-1 & 0
\end{array}
\right)
$$
as $\epsilon(x;1,0)=U(x)$ and $\epsilon(x;0,1)=V(x)$, so the symplectic form on $L^\omega$ is $du\w dv$ in this frame.}
Write $M(t):=[d_{x_0}Fl^{X_H}_t]$ for the $4\times 4$ time $t$-linearized flow, and put the symplectic frame $U,V$ into the $4\times 2$-frame matrix $F(x)=[\,U(x)\ V(x)\,]$.
{Writing $x_t=Fl^{X_H}_{t}(x_0) $, we have
\begin{equation}
\label{eq:reduced_monodromy}
\Psi(t)=\Omega_{2}^T F( x_t )^T \Omega_{4} M(t) F(x_0) \in Sp(2).
\end{equation}
To explain this formula, we push forward the $U(x_0),V(x_0)$-frame, encoded in the matrix $F(x_0)$,  using the time $t$-linearized flow $M(t)$. Then we take the coefficients of this transported frame in terms of the $U(x_t),V(x_t)$-frame. This is done by taking the symplectic inner product $\Omega_{4}$ with $U(x_t),V(x_t)$-frame $F(x_t)$ and applying an additional rotation $\Omega_{2}^T$. }

With this matrix description, we have the following well-known characterization.
\begin{lemma}
    \label{lemma:transverse_nondegeneracy}
    A periodic Hamiltonian orbit of $X_H$ in $\R^4$ with period $T_0$ is non-degenerate if and only if $\mathrm{tr}(\Psi(T_0)) \neq 2$. 
\end{lemma}

\subsubsection{Degree definition of Conley-Zehnder index}
\label{sec:degree_def_CZ}
We define the Conley-Zehnder index following Salamon-Zehnder's paper, \cite{salamon_zehnder_1992}, but specialize to the $2$-dimensional case in order to get more explicit formulas.
{We will use the symplectic group $Sp(2)$ consisting of $2\times 2$ symplectic matrices and the unitary group $U(1)$. We will identify $e^{i\phi} \in U(1)$ with the symplectic matrix
$$
\left(
\begin{array}{cc}
\cos(\phi) & -\sin(\phi) \\
\sin(\phi) & \cos(\phi)
\end{array}
\right)
.
$$
}
The Maslov cycle is {defined as} the set 
$$
V=\{ \psi \in Sp(2)~|~\det(\psi-\id_{2})=0 \} .
$$
\begin{proposition}
    The complement $Sp(2)\setminus V$ consists of two path-connected components. These components can be characterized as
    $$
    V_+ =\{ \psi \in Sp(2)~|~\det(\psi-\id_{2})>0 \} =\{ \psi \in Sp(2)~|~\mathrm{tr}(\psi)< 2 \} ,
    $$
    and
    $$
    V_- =\{ \psi \in Sp(2)~|~\det(\psi-\id_{2})<0 \} =\{ \psi \in Sp(2)~|~\mathrm{tr}(\psi)> 2 \} . 
    $$
\end{proposition}
For $\psi\in Sp(2)$, define
$$
\rho(\psi)=(\psi \psi^T)^{-1/2} \psi \in U(1).
$$
This map $\rho$ defines the end point of a deformation retract, which we can write out more explicitly in our $2$-dimensional situation by the following lemma.
\begin{lemma}
    Given a $2\times 2$ symplectic matrix
    $$
    \Psi=\left(
    \begin{array}{cc}
    a & b \\
    c & d
    \end{array}
    \right) 
    \in Sp(2)
    ,
    $$
    we have the formula
    \begin{equation}
    \label{eq:explicit_contraction}
    \rho(\Psi)=
    (\Psi \Psi^T)^{-1/2} \Psi
    =
    \frac{1}{\sqrt{(a+d)^2 +(b-c)^2}}
    \left(
    \begin{array}{cc}
    a+d & b-c \\
    -b+c & a+d
    \end{array}
    \right)
    .
    \end{equation}
\end{lemma}
\begin{proof}
We apply the formula for the square root of a $2\times 2$ matrix~\cite{levinger_square_1980} to $M:=\Psi \Psi^T$. 
{This formula is as follows. Setting $s=\pm \sqrt{\det M}$ and $t^2=\mathrm{tr}(M)+2s$, a square root of $M$ is given by
$$
R=\frac{1}{t}(M+s \id_2).
$$
Plugging in 
$$
M=
\left(
\begin{array}{cc}
a^2+b^2 & ac+bd \\
ac+bd & c^2 +d^2
\end{array}
\right)
$$
and inspecting the signs, we find $s=1$ and $t=\sqrt{(a+d)^2 +(b-c)^2}$ for the positive definite root. So we have}
$$
(\Psi \Psi^T)^{1/2} = 
\frac{1}{\sqrt{(a+d)^2 +(b-c)^2}}
\left(
\begin{array}{cc}
a^2 + b^2 + 1 & ac + bd \\
ac + bd & c^2 + d^2 + 1
\end{array}
\right)
.
$$
By taking its inverse and multiplying with $\Psi$ on the right we get the desired formula.
\end{proof}

Let $\psi:[0,1] \to Sp(2)$ be a non-degenerate path of symplectic matrices, meaning $\psi(1)$ has no eigenvalue equal to $1$. We can extend $\psi$ to a path $\tilde \psi : [0,2] \to Sp(2)$ whose endpoint is either $W_+=-\id_2$ or $W_-=\mathrm{diag}(2,1/2)$, where we require the extension to be disjoint from the Maslov cycle. The choice of the endpoints ensure that $\rho(\tilde \psi(2)) =\pm 1$.
\begin{definition}
    The {\bf Conley-Zehnder index} of $\psi:[0,1]\to Sp(2)$ is defined by $$\mu_{CZ}(\psi):=\deg(\rho(\tilde \psi)^2),$$ where $\tilde \psi$ is the extension described above.
{    The {\bf (transverse) Conley-Zehnder index of a Hamiltonian orbit} $\gamma$ in $\R^4$ is defined as
    $$
    \mu_{CZ}(\gamma):=\mu_{CZ}(\Psi).
    $$
    where $\Psi$ is the transverse linearized path of symplectic matrices as defined in Equation~\eqref{eq:transverse_linearized_path}.}
\end{definition}

{
To give some intuition behind this definition, note that $\tilde \psi(2)$ lies in one of the two components of $Sp(2)\setminus V$. After retraction to the circle $U(1)$, we have $\rho(\tilde \psi(2)) =\pm 1$. By doubling the speed of this path $t\mapsto \rho(\tilde \psi(t))$, we obtain a loop $t\mapsto  \rho(\tilde \psi(t))^2$, starting and ending at $1$, and the Conley-Zehnder index is the winding number of this loop.

\begin{remark}
In the setup we consider here, i.e.~the special topology of the energy hypersurface, the index of an orbit does not depend on the choice of a global symplectic trivialization. Moreover, the index of a family of non-degenerate periodic Hamiltonian orbits is constant.
\end{remark}
}

In what follows, we show that the Conley-Zehnder index can be computed in the $2$-dimensional case without choosing an explicit extension.
{
Namely, in Proposition~\ref{prop:index_computation} we show that the index can be determined by the amount of rotation of the path $t\mapsto \rho(\psi(t))$ around $S^1$.
Figure~\ref{fig:rotation_angle} provides a schematic illustration of this result.
}
\begin{lemma}
\label{lemma:fibers}
The fiber $\rho^{-1}(1) \setminus \id_2 $ lies in the component $V_-$ and the fiber $\rho^{-1}(-1)$ lies in the component $V_+$.
\end{lemma}
\begin{proof}
Take 
$$
\Psi=\left(
\begin{array}{cc}
a & b \\
c & d
\end{array}
\right)
\in \rho^{-1}(\pm 1)
,
$$
By formula~\eqref{eq:explicit_contraction} either fiber satisfies $b=c$. With this restriction, we find $ad-b^2=1$, and hence 
    $$
    \mathrm{tr}(\Psi) = a+\frac{1+b^2}{a}.
    $$
For $\Psi \in \rho^{-1}(1)$, we have $a>0$, and hence
$$
\mathrm{tr}(\Psi) = a+\frac{1+b^2}{a}\geq a +\frac{1}{a} \geq 2 
$$
as the minimum is attained in $a=1$. Similarly, for $\Psi \in \rho^{-1}(-1)$, we have $a<0$, and we have $\mathrm{tr}(\Psi)\leq -2$ in that case.
The claim follows.
\end{proof}

\begin{proposition}
    \label{prop:index_computation}
    Let $\psi : [0, T] \to Sp(2)$ be a non-degenerate path of symplectic matrices. Lift the path $t\mapsto\rho(\psi(t))$ to a path $\theta:[0, T] \to \R$ in the universal cover, imposing $\theta(0)=0$. Then,
    $$\mu_{CZ}(\psi)=\begin{cases}
        \text{closest even integer to }\theta(T)/\pi, & \text{if }\mathrm{tr}(\psi(T)) > 2 \\
        \text{closest odd integer to }\theta(T)/\pi, & \text{if }\mathrm{tr}(\psi(T)) < 2 
    \end{cases}$$
\end{proposition}

\begin{proof}
If $\mathrm{tr}(\psi(T)) > 2$, we can extend $\psi$ to a path $\tilde\psi:[0, 2T]\to Sp(2)$ such that $\tilde\psi([T, 2T]) \subset V_-$. By Lemma~\ref{lemma:fibers}, we have $\rho(\tilde \psi(t)) \neq -1$ for $t\in [T, 2T]$, while $\rho(\tilde \psi(2T)) = 1$. Thus, $\deg(\rho(\tilde \psi)^2)$ is equal to the closest even integer to $\deg(\rho(\psi)^2)$.

The case $\mathrm{tr}(\psi(T)) < 2$ can be proved analogously.
\end{proof}
{
\begin{figure}[!t]
    \centering
    \includegraphics[width=0.8\textwidth]{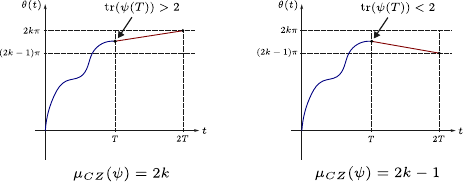}
    \caption{Schematic plot of the lift $\theta$ of the extension $\rho(\tilde \psi(t))$ to the cover $\R \to S^1=U(1)$ for two cases, along with their corresponding Conley-Zehnder indices.}
    \label{fig:rotation_angle}
\end{figure}
}
\subsection{Symmetries}
The Hamiltonian $H$ of the restricted three-body problem is invariant under the anti-symplectic involution
$$
r:(q_1,q_2;p_1,p_2) \longmapsto (q_1,-q_2;-p_1,p_2).
$$ 
In complex coordinates, this involution is given by $r(q;p)=(\bar q,-\bar p)$.
The symmetry of the Jacobi Hamiltonian extends to the Levi-Civita regularization in the following way.

There are two anti-symplectic involutions corresponding to $r$.
These are
$$
R_+:(z,w)\longmapsto (\bar z,-\bar w)
\quad
R_-:(z,w)\longmapsto (-\bar z,\bar w).
$$
The union of the fixed point loci of these involution maps cover the fixed point locus of $r$.

It is known that the fixed point locus of an anti-symplectic involution is a Lagrangian submanifold. 
In our case, we see this directly and explicitly, as we have
$$
Fix(R_+)=\{ (z_1,0;0,w_2) \}
\text{ and }
Fix(R_-)=\{ (0,z_2;w_1,0) \}
$$
Both of these sets are $2$-planes in $\C^2\cong \R^4$. Applying Proposition~\ref{prop:starshaped}, we see that $C_+:=Fix(R_+) \cap \Sigma_{\mu,c}$ forms a single circle. The same holds for $C_-:=Fix(R_-) \cap \Sigma_{\mu,c}$.
Using the starshapedness again, we see that the following holds.
\begin{proposition}
The sets 
$$
S_+ := \{ (z_1,z_2;w_1,w_2) \in \Sigma_{\mu,c}~|~z_2=0 \} 
\text{ and }
S_- := \{ (z_1,z_2;w_1,w_2) \in \Sigma_{\mu,c}~|~z_1=0 \} 
$$
are both diffeomorphic to a $2$-sphere and contain the closed curves $C_+$ and $C_-$, respectively.
\end{proposition}

\subsubsection{Symmetric periodic orbits}
Consider a periodic orbit $x\in C^\infty([0,T],\R^4)$ of $X_H$ with period $T$. The ``mirrored'' orbit
$$
x_r(t)=r\circ x(T-t)
$$
is then also periodic with period $T$.
\begin{definition}
We call a periodic orbit $x$ {\bf symmetric} if $x_r(t)=x(t)$.
\end{definition}
It follows from this definition that a (regular) symmetric periodic orbit intersects $Fix(r)$ at least twice, namely at $t=0$ and $t=T/2$.

\subsubsection{Parametrizing the fixed point locus}
\label{sec:parametrizing_fixed_locus}
In the proof we need explicit parametrizations of the fixed point locus of these involutions.
To do so, we consider the intersection of the fixed point locus with the positive $z_i$-axis for $i=1,2$, and denote the $z_i$ coordinate of the intersection points by $B^i_{\mu,c}$. These correspond to points in the boundary of the Hill's region.

We can parametrize $C_+=Fix(R_+)\cap \Sigma_{\mu,c }$ by setting
$$
w_2^{\pm}(z_1) = -(2 z_1^3-\mu z_1)\pm \sqrt{(2 z_1^3-\mu z_1)^2-2(c z_1^2 -\frac{\mu z_1^2}{|1-2z_1^2|} -\frac{1-\mu}{2})}
$$
for $z_1\in [-B^1_{\mu,c}, B^1_{\mu,c}]$.

For $C_-=Fix(R_-)\cap \Sigma_{\mu,c }$, we find analogously
$$
w_1^{\pm}(z_2) = -(-2 z_2^3-\mu z_2)\pm \sqrt{(-2 z_2^3-\mu z_2)^2-2(c z_2^2 -\frac{\mu z_2^2}{|1+2z_2^2|} -\frac{1-\mu}{2})}
$$
for $z_2\in [-B^2_{\mu,c}, B^2_{\mu,c}]$.

\subsubsection{Crossing number}
Before we give a precise definition, we first define the crossing number intuitively.
Given a periodic orbit $\gamma$ of $X_H$, we want to define the crossing number of the (set-theoretic) intersection number of $\gamma$ with the set $q_2=0$.
This naive definition would be undefined because of collision orbits, so we use a slightly more complicated definition.
First of all, we observe that any periodic orbit $\gamma$ in the Moser regularization lifts to either to a periodic orbit in the Levi-Civita regularization or to an orbit whose initial and final points are antipodal.

\begin{definition}
    Suppose $\gamma$ is a periodic orbit of the Moser regularized vector field of period $T$. Denote its lift to the Levi-Civita regularization by $\tilde \gamma$.
    Then the {\bf crossing number} of $\gamma$ is 
    $$
    cross(\gamma) = \# \{ t\in [0, T)~|~pr_{z_1}\circ \tilde \gamma(t)=0 \text{ or } pr_{z_2}\circ \tilde \gamma(t)=0 \}.
    $$
\end{definition}
We note that the crossing number counts the number of intersections with $S_\pm$.

\subsection{Validated numerics}
\label{sec:validated_numerics}
In this work, all numerical computations are done rigorously using interval arithmetic, which gives guaranteed bounds on the computed results. A brief overview of interval arithmetic is given below, but for more details we refer the readers to~\cite{moore_introduction_2009}. 

The key idea of interval arithmetic is to do computations with sets. We consider the set of (closed) real intervals.
$$ \mathbb{I}\R = \{ [a] = [a_l, a_r] ~ | ~ a_l, a_r \in \R, a_l \leq a_r \} $$
For each of the basic operations $\circ = +, -, \times, /$, we define
$$ [a] \circ [b] = \{ x \circ y ~ | ~ x \in [a], y \in [b] \} $$
By defining arithmetic this way, it is guaranteed that the resulting interval $[a] \circ [b]$ encloses all possible values resulting from the arithmetic operation. For a computer implementation involving floating point numbers, we extend real interval arithmetic by rounding the endpoints of $[a] = [a_l, a_r]$ to the nearest representable numbers\footnote{By representable number we mean a real number that can be represented exactly by a floating point number.} $\tilde a_l, \tilde a_r$ so that $[a] \subset [\tilde a_l, \tilde a_r]$. This is known as outward rounding and can be extended to all basic operations. 
\begin{definition}
An {\bf interval vector} or {\bf box} is an element $[v]$ in $\mathbb{I}\R^n$. If $x=(x_1,\ldots,x_n) \in [v]$ we refer to $[v]$ as an {\bf enclosure} of $x$.
\end{definition}
A basic application of interval arithmetic is that it can be used to compute explicit bounds on a continuous function over a compact domain in $\R^n$. Given an interval vector $[u]$, we shall denote by $[f([u])]$ an interval vector enclosing the image $f([u])$.

In our setting, we apply interval arithmetic to obtain rigorous enclosures of solutions to ODEs and their derivatives. For this, we use the CAPD Library~\cite{kapela_capd_2021}, which implements the $C^1$-Lohner algorithm~\cite{zgliczynski_c1-lohner_2002} based on the Taylor method for solving differential equations. To be explicit, let $f:\R^n \to \R^n$ be a real analytic vector field and consider the initial value problem
\begin{equation*}
\left\{
\begin{aligned}
\dot x(t) & = f(x(t)), \\
x(0) & = x_0.
\end{aligned}
\right.
\end{equation*}
The algorithm provides rigorous bounds for the trajectory $Fl_t^f([x_0])$ as well as for the linearized flow $dFl_t^f([x_0])$. Moreover, for a given Poincar\'e section, rigorous enclosures for the Poincar\'e map $P([x_0])$ and its derivative $dP([x_0])$ can be computed as well. For details of the algorithm, see~\cite{zgliczynski_c1-lohner_2002, wilczak_Cr_Lohner_2011}.

{
There is extensive literature on applying these methods to validate dynamical properties of low-dimensional ODE systems. Most relevant to our work are studies on symmetric periodic orbits in Hamiltonian systems, such as the existence and stability of periodic orbit families in the H\'enon-Heiles system by~\cite{barrio_rodriguez_systematic_2014}. For the restricted three-body problem, prior works have examined the existence and bifurcations of Lyapunov orbits~\cite{capinski_lyapunov_2012}, halo orbits~\cite{walawska_touch_and_go_2019}, and ejection-collision orbits~\cite{capinski_collision_orbits_2023}.
}

\section{Proofs}
\label{sec:proofs}

\subsection{Crossing theorem}
Before we dive into the proofs of the theorems stated in the introduction, we first state a result which describes the orbits more geometrically and which is computationally easier.
\begin{theorem}
\label{thm:crossing_axis}
Define $V = V_h \cup V_v$, with
\begin{align*}
    V_h & = \{ ( \mu, c) ~|~\mu \in [0, 0.9] \text{ and } c\in [2.1, 2.1 + 10^{-5}] \}, \\
    V_v & = \{ ( \mu, c) ~|~\mu \in [0,10^{-4}] \text{ and } c\in [1.52, 2.1] \} .
\end{align*}
Consider $(\mu,c) \in V$. 
\begin{enumerate}
    \item The component $\Sigma_{\mu,c}$ has a unique symmetric retrograde periodic orbit with crossing number $2$, and this orbit is transversely non-degenerate.
    \item The component $\Sigma_{\mu,c}$ has a symmetric direct periodic orbit with crossing number $2$. If $(\mu,c) \in V_h$, this orbit is unique and transversely non-degenerate.
    \item If $(\mu,c) \in V_h$, the retrograde and the direct orbits described above are the only two symmetric periodic orbits with crossing number $2$.
\end{enumerate}
These orbits project to non-contractible orbits in $\R P^3$.
\end{theorem}

Our argument follows Birkhoff's original argument, but it works on the regularization and is quantitative. 
We construct a cover of $C_+$ by finitely many closed intervals $I_p$, indexed by an interior point $p$, such that one of the following holds:
\begin{enumerate}
    \item[($-$)] the first hitting time of $I_p$ with $S_-$ defines a continuous function $\tau_p:I_p \to \R_{>0}$. Furthermore, for all $q\in I_p$, the flow line $Fl^X_t(q)$ intersects $S_-$ transversely at $t=\tau_p(q)$.
    \item[($+$)] the first hitting time of $I_p$ with $S_+$ defines a continuous function $\tau_p:I_p \to \R_{>0}$. Furthermore, for all $q\in I_p$, the flow line $Fl^X_t(q)$ intersects $S_+$ transversely at $t=\tau_p(q)$.
\end{enumerate}
Logically, the two options are not disjoint, but both options need to be considered in order to cover the entire fixed point locus. While the existence of such a covering is not given a priori, such a covering can be obtained numerically through subdivision where for each interval the transversality condition is validated.

We denote the first hitting point of $q\in I_p$ with $S_\pm$ in either case by $h_p(q)$. In the implementation, this is computed using the class \texttt{IPoincareMap} provided by the CAPD library. 
To find symmetric orbits, we consider the slope functions $s_p$ defined on each interval $I_p$ as follows: 
\begin{enumerate}
    \item[($-$)] if $h_p:I_p\to S_-$, then define $s_p=pr_{w_2}\circ h_p$. Note that $h_p(q)\in C_-$ if and only if $s_p(q)=0$.
    \item[($+$)] if $h_p:I_p\to S_+$, then define $s_p=pr_{w_1}\circ h_p$. Note that $h_p(q)\in C_+$ if and only if $s_p(q)=0$.
\end{enumerate}
These slope functions are continuous by construction, and its zeros correspond to symmetric periodic orbits as shown by the following lemma.
\begin{lemma}
Suppose that $I_p\subset C_+$ is an interval as above, and assume that the slope function $s_p$ has a zero at $p_0\in I_p$.
Then $p_0$ is the initial point of a symmetric periodic orbit.
\end{lemma}

\begin{proof}
Assume that $p_0 \in I_p$ with $s_p(p_0)=0$, so there is an orbit segment $\gamma$ connecting $p_0$ to a point $q_0\in C_\pm$.
    For simplicity, we consider the case of an orbit segment to $q_0\in C_+$.
    Define
$$
\delta(t) = R_+\circ \gamma(-t) \text{ for } t\in [-\tau_p(p_0),0].
$$
Then $\delta(0)=\gamma(0)$ and $\delta(-\tau_p(p_0) ) =\gamma(\tau_p(p_0) )$. Furthermore, we have
$$
\dot \delta(t) = X_H\circ \delta(t),
$$
because of the chain rule and fact that $R_+$ is an anti-symplectic involution leaving $H$ invariant.
By uniqueness of solutions of ODEs, the orbit segments $\gamma$ and $\delta$ glue together to a periodic orbit with period $2\tau_p(p_0)$. 

The case of an orbit segment connecting $p_0$ to a point $q_0\in C_-$ follows in a similar way, but in this case both involutions $R_\pm$ need to be used. Specifically, the orbit segments $R_+\circ \gamma,\gamma,R_-\circ \gamma, R_-\circ R_+ \circ \gamma$ glue together to a periodic orbit with period $4\tau_p(p_0)$. 
\end{proof}

Finally, we refine the covering if necessary so that for each $I_p$, either one of the following is validated:
\begin{enumerate}
    \item $s_p$ has no zero in $I_p$, in which case there is no point $q \in I_p$ that is the initial point of a symmetric periodic orbit with crossing number $2$;
    \item $s_p$ has a zero in $I_p$, which we validate by checking that it has opposite signs at the two endpoints. This proves existence of a symmetric periodic orbit. 
\end{enumerate}
In the second case, we can further verify local uniqueness by computing the derivative of the slope function. Namely, we verify that the sign of the derivative of the slope function is constant on the entire interval $I_p$. 
The upshot of the construction is the following:
\begin{itemize}
    \item We obtain a cover of $C_+$ consisting of intervals $I_p$, for which the existence or non-existence of a zero of $s_p$ is validated.
    \item The number of intervals $I_p$ containing a zero of $s_p$ depends on the parameters $c$ and $\mu$.
\end{itemize}

\begin{remark}
We note that we saved on computation effort by using the additional symmetry corresponding to \eqref{eq:sympl_involution}. In short, we only construct half a cover of $C_+$; the remaining intervals are obtained by applying $s$.    
\end{remark}

\begin{remark}
    In the implementation, we first choose an interval which contains an approximate location $p_0\in C_\pm$ of a symmetric orbit as its midpoint, and include it as one of the intervals $I_p$ during the construction. The approximate location is either computed using a combination of grid search and Newton's method, or in the case of small $\mu$, the locations of the circular orbits of the rotating Kepler problem are used. This helps prevent alignment issues.
\end{remark}

The analogous construction for $C_-$ is considered as well. By considering these coverings we are guaranteed to find all potential symmetric periodic orbits with crossing number $2$, because for such orbits the intersections with $S_\pm$ always happen at the fixed point locus $C_\pm$.

\begin{proof}[Proof of Theorem~\ref{thm:crossing_axis}]
The parameter range $V$ is subdivided into small overlapping boxes $[\underline{\mu_i}, \overline{\mu_i}] \times [\underline{c_j}, \overline{c_j}]$, and we verify the above construction for each box.
For an interval $I_p$ containing a zero of $s_p$, we shoot until the first hitting point on $S_\pm$. Between time steps we enclose the trajectory and verify using Equation~\eqref{eq:retro_condition} the orbit to be either retrograde or direct.
The existence and uniqueness statements follow directly from the construction.

Note that all orbits found by this method intersect both $C_+$ and $C_-$. It follows that all symmetric periodic orbits with crossing number $2$ for $\mu, c \in V$ are non-contractible in $\R P^3$.

To verify transverse non-degeneracy, suppose $p_0\in I_p \subset C_+$ is a zero of $s_p$. We compute an enclosure of the first hitting time $\tau_p(p_0)$ with $C_-$. Doubling this, we obtain an enclosure for the duration $T_{p_0}$ of an orbit which projects to a periodic orbit in the unregularized problem or in the Moser regularization. Another orbit of interest is its double cover, which has minimal period $2T_{p_0}$ in the Levi-Civita regularization. We verify non-degeneracy for both cases separately.

To proceed, recall the $4\times 2$-frame matrix 
$F(x)=[\,U(x)\ V(x)\,]$,
where $U$ and $V$ are the vectors defined in Equation~\eqref{eq:frame_contact_structure}.
We compute the enclosure for the $4 \times 4$ monodromy matrix $M(T_{p_0})$, and apply the formula \eqref{eq:reduced_monodromy} for the reduced monodromy:
$$
\Psi(T_{p_0})=\Omega_{2}^T F( Fl^{X_K}_{T_{p_0}}(p_0)   )^T \Omega_{4} M(T_{p_0} ) F(p_0) \in Sp(2),
$$
where 
and $\Omega_2$ and $\Omega_4$ are matrix representations of the symplectic forms on $\R^2$ and $\R^4$, as defined in Section~\ref{sec:frames_invariants}.
Taking all relevant enclosures, we obtain an enclosure of $\Psi(T_{p_0})$. Non-degeneracy can then be shown by verifying $\det(\Psi(T_{p_0}) - \id_{2}) \neq 0$, or equivalently $\mathrm{tr}(\Psi(T_{p_0} ))\neq 2$.
\end{proof}

\begin{remark}
The retrograde property ceases to hold for all Jacobi energies, as indicated by Figure~\ref{fig:retrograde}. For example, for $\mu=0.9$ and $c=1.2$ (which is well above all critical values), the orbit found by this shooting method will slow down so much, that its apparent movement can appear direct.
\end{remark}

\begin{proof}[Proof of Theorem~\ref{thm:direct_existence}]
The parameter range is subdivided into overlapping boxes in a way that for each $\mu \in [0.1, 0.5]$, the minimum Jacobi energy $c$ such that $(\mu,c)$ is contained in one of the boxes is validated to be above the critical level. For each box, we take an interval $I_p \subset C_+$ which contains an approximate location of the direct orbit as its midpoint. We verify existence by showing $s_p$ has a zero in $I_p$. The transverse non-degeneracy and the direct properties are validated in the same way as in the proof of Theorem~\ref{thm:crossing_axis}.
\end{proof}

\subsection{Action theorems}
We will prove the following theorems concerning symmetric periodic orbits with an action bound.

\begin{theorem}[Uniqueness of retrograde orbit as a symmetric orbit]
\label{thm:action_retro}
Define $V = V_h \cup V_v$, with
\begin{align*}
    V_h & = \{ ( \mu, c) ~|~\mu \in [0, 0.5] \text{ and } c\in [2.1, 2.1 + 10^{-6}]  \} , \\
    V_v & = \{ ( \mu, c) ~|~\mu \in [0,10^{-5}] \text{ and } c\in [1.62, 2.1]  \} .
\end{align*}
Then for all $(\mu,c)\in V$ there is $A_{\mu,c}>0$ and a retrograde symmetric periodic orbit $\gamma_{\mu,c}$ such that every symmetric periodic orbit with action less than $A_{\mu,c}$ is a reparametrization of $\gamma_{\mu,c}$. 
Moreover, the orbit $\gamma_{\mu,c}$ is non-degenerate.
\end{theorem}
In short, in this parameter range, the retrograde orbit is the unique symmetric action minimizer.

\begin{theorem}[Uniqueness of small action symmetric orbits]
\label{thm:action_small} 
For all $\mu \in [0, 0.5]$ and $c\in [2.1, 2.1 +10^{-6}]$ there is $A_{\mu,c}>0$, a retrograde symmetric periodic orbit $\gamma^r_{\mu,c}$ and a direct symmetric periodic orbit $\gamma^d_{\mu,c}$ such that every symmetric periodic orbit with action less than $A_{\mu,c}$ is a reparametrization of $\gamma^r_{\mu,c}$ or of $\gamma^d_{\mu,c}$. 
Moreover, the orbits $\gamma^r_{\mu,c}$ and $\gamma^d_{\mu,c}$ are non-degenerate.
\end{theorem}

In order to compute enclosures of the action of the trajectories, we augment the Hamiltonian vector field by including the action $A$ as an additional variable. In other words, we integrate the augmented vector field
\[
\bar X_H =
\sum_{j=1}^2\left(\frac{\partial H}{\partial w_j}\frac{\partial }{\partial z_j}
-\frac{\partial H}{\partial z_j}\frac{\partial }{\partial w_j} \right)
+
 \sum_{j=1}^2 \left(w_j \frac{\partial H}{\partial w_j} + z_j \frac{\partial H}{\partial z_j} 
\right)
\frac{\partial}{\partial A}
.
\]
The last component is simply $(wdz-zdw)(X_H)$.

We first describe a construction which allows for a systematic search for all symmetric orbits below a given action bound $A_0$. The construction is based on the following observation: a symmetric orbit starting from a fixed point locus will arrive at the same fixed point locus at half period with half action, due to its symmetry. Hence, we shoot from a fixed point locus until the action exceeds $A_0/2$, and check whether the trajectory intersects the corresponding fixed point locus.

To proceed, we construct a cover of $C_+$ consisting of intervals $I_p$, together with a sequence of maps such that the following holds:
$$
I_p \stackrel{h_p^0}{\longrightarrow} I^1_p \stackrel{h_p^1}{\longrightarrow} I^2_p \stackrel{h_p^2}{\longrightarrow} I^3_p \stackrel{h_p^3}{\longrightarrow} \ldots
,
$$
where $I^j_p\subset S_+$ and $h_p^j$ are return maps to $S_+$ for each $j$, as implemented by the class \texttt{IPoincareMap} of the CAPD library.
For $j=0$, we also allow the possibility of $h_p^0$ being a composite $h_p^0 = h_p^{0,+} \circ h_p^{0,-}$, where $h_p^{0,-}$ is given by the first hitting point with $S_-$ and $h_p^{0,+}$ with $S_+$. This is necessary to ensure transverse intersections with $S_\pm$ at each step. For instance, if $I_p$ contains a point of the boundary of the Hill's region, the first intersection with $S_+$ will not be transverse. 
Note that if a trajectory starting at $C_+$ returns to $C_+$ before intersecting $S_-$, it will double to a periodic orbit which does not intersect $S_-$. Hence, by relaxing the condition on $h_p^0$ as above, we do not miss out on possible intersections with $C_+$.

Denoting the hitting times by $T_p^j$, we have the following relation:
$$
h_p^j \circ  h_p^{j-1} \circ \ldots \circ h_p^0 (q)
=
pr_{z_1,z_2;w_1,w_2}\circ Fl^{\bar X_H}_{T_p^j(q)}(q).
$$
Define the slope and action functions on each interval $I_p$ as
$$
A_p^j(q)
=
pr_{A}\circ Fl^{\bar X_H}_{T_p^j(q)}(q),
\qquad
s_p^j(q)
=
pr_{w_2}\circ Fl^{\bar X_H}_{T_p^j(q)}(q).
$$
For each $I_p$ we find the smallest $j_0$ such that $A_p^{j_0}(I_p) > A_0/2$. We obtain a sequence $I_p^j$ for $j=1,\ldots, j_0-1$ and look for zeros of the function $s_p^j$. The upshot is that any $R_+$-symmetric orbit with action less than $A_0$ will appear as a zero of $s_p^j$ for some $p$ and $j$.

\begin{proof}[Proofs of Theorems~\ref{thm:action_retro}~and~\ref{thm:action_small}] 
    Given an enclosure of the parameters $\mu$ and $c$, we first find enclosures for the location of the retrograde and direct orbits on which the existence of the orbits has been validated. By integrating the augmented vector field $\bar X_H$ starting from these enclosures, we obtain enclosures for the action of both orbits. We use the largest value $A_0$ of these enclosures as the action bound and verify the above construction.
    
    For Theorem~\ref{thm:action_retro}, we verify the construction with action bound $A_0$ corresponding to the retrograde orbit, and show that $s_p^j$ has a zero only at the interval which contains the retrograde orbit. For Theorem~\ref{thm:action_small}, the action bound $A_0$ corresponding to the direct orbit is used and we show $s_p^j$ has zeros only at the two intervals containing the retrograde and direct orbits, respectively.

    The non-degeneracy statement follows directly from Theorem~\ref{thm:crossing_axis}.
\end{proof}

\subsection{Validated index computations}
\label{sec:validated_index_computations}
Since the computation of the Conley-Zehnder index is an important topic in symplectic dynamics, we will show that this index can be computed directly without access to a family. 
{As an application, we can use the index to validate the existence of a bifurcation.}
For the sake of explicitness, we consider a specific example. 
\begin{theorem}
\label{thm:index}
    Consider $\mu=0.99$. There is a family of symmetric periodic orbits $\gamma_c$ with $c\in [1.57617, 1.57626]$ such that 
    \begin{itemize}
        \item $\gamma_c$ is non-degenerate
        \item $\gamma_c^2$ non-degenerate for $c=1.57617$ and for $c=1.57626$.
        \item the Conley-Zehnder index satisfies
        \[
        \mu_{CZ}(\gamma_c)=3,\quad 
        \mu_{CZ}(\gamma_c^2) =\begin{cases}
            5 & c=1.57626 \\
            6 & c=1.57617
        \end{cases} 
        \]
    \end{itemize}
\end{theorem}

The orbit family is shown in Figure~\ref{fig:index_family}. In this situation, the family of orbits $\gamma_c$ is non-degenerate, but it goes through a periodic doubling bifurcation; its double cover goes from elliptic to positive hyperbolic, causing the index jump. The phase portrait of the Poincar\'e map with section $q_2=0$ before and after the bifurcation is shown in Figure~\ref{fig:phase_portrait}.
{In Remark~\ref{rem:bifurcation_floer}, we will explain how the existence of a bifurcation can be deduced from the index change.} 

\begin{figure}[!t]
    \centering
    \includegraphics[width=0.5\textwidth]{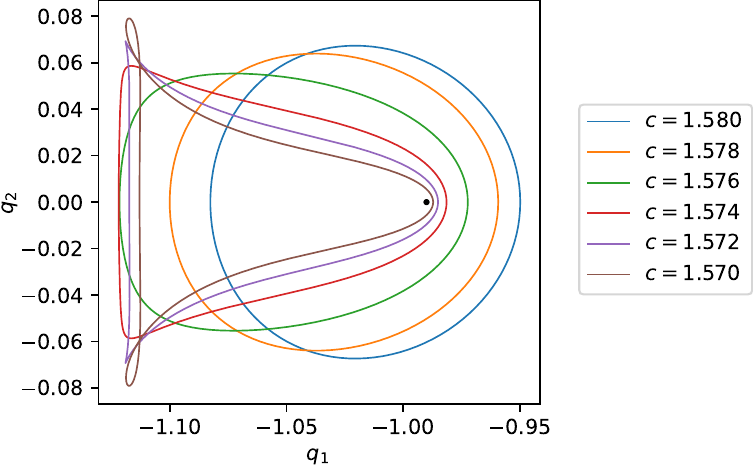}
    \caption{Family of periodic orbits for $\mu=0.99$ which goes through a period doubling bifurcation.}
    \label{fig:index_family}
\end{figure}

\begin{figure}[!t]
    \centering
    \includegraphics[width=0.8\textwidth]{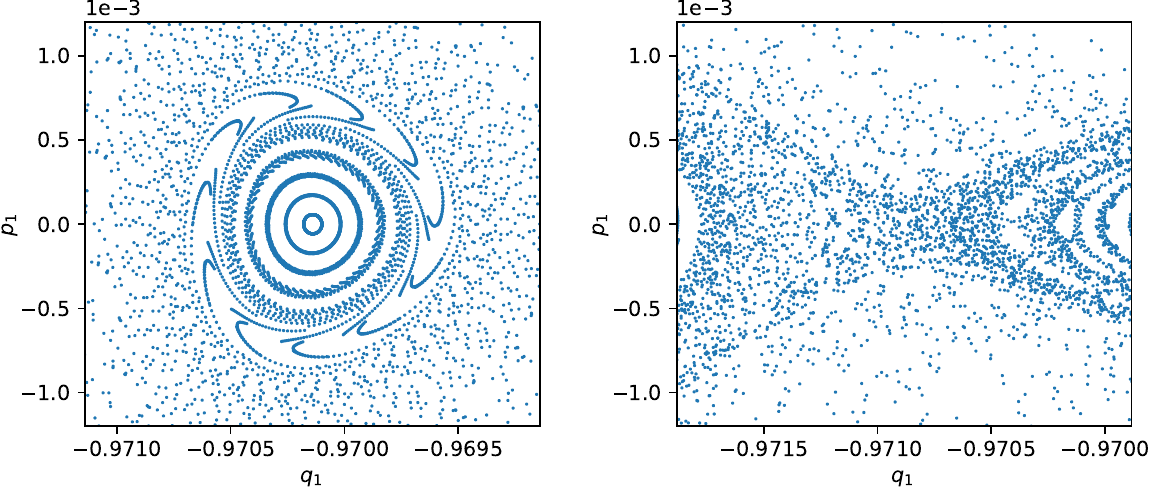}
    \caption{Phase portrait of the Poincar\'e map  with section $q_2=0$ in a neighborhood of the periodic orbit family of Theorem~\ref{thm:index}. The phase portrait before (left) and after (right) the period doubling bifurcation show the transition from an elliptic to a hyperbolic orbit.}
    \label{fig:phase_portrait}
\end{figure}

In what follows, we describe a general procedure for validated computation of the Conley-Zehnder index. The proof of Theorem~\ref{thm:index} will be a direct application of the procedure to the specific orbit and energy levels.

We first find enclosures of the initial condition $[q_0] \subset \C^2$ and period $[T_0] \subset \R$ of a non-degenerate, symmetric, periodic orbit $\gamma_0$.
We get a sequence of overlapping interval vectors $[q_j] \subset \C^2$, $[t_j] \subset \R$ and $[M_j]\subset Mat_{4\times 4}$ for $j=0, \ldots, N$ with $[0,T_0] \subset \cup_j [t_j]$ such that
\begin{itemize}
    \item $Fl^{X_H}_{[t_j]}(q_0) \subset [q_j]$. In particular, $\cup_j [q_j]$ is a neighborhood of the periodic orbit $\gamma_0$.
    \item $d_{q_0}Fl^{X_H}_{[t_j]} \subset [M_j]$
\end{itemize}
As before, define a symplectic frame
$F(x)=[\,U(x)\ V(x)\,]$ which we also enclose along the orbit.
Compute an enclosure at each time step for the path of symplectic matrices
$$
t \mapsto \Psi(t) = \Omega_{2}^T F( Fl^{X_K}_{t}(q_0)   )^T  \Omega_{4} M(t) F(q_0),
$$
which we write as $[\Psi([t_j])]$. Finally, compute the retraction $\rho$ to $U(1)$ using the explicit formula~\eqref{eq:explicit_contraction}, and denote the resulting enclosures as $[\rho(\Psi([t_j]))]$ for each $j$.

By choosing the time step smaller if necessary, we ensure that each enclosure $[\rho(\Psi([t_j]))]$ is contained in one of the four half-planes
\begin{align*}
U_1 = \{ (x,y) \in \R^2 ~|~x>0 \},
\qquad
& U_2 = \{ (x,y) \in \R^2 ~|~y>0 \},
\\
U_3 = \{ (x,y) \in \R^2 ~|~x<0 \},
\qquad
& U_4 = \{ (x,y) \in \R^2 ~|~y<0 \} .
\end{align*}
In addition, suppose the non-degeneracy of the orbit has been validated at the final time interval, i.e. $[\mathrm{tr}(\Psi([t_N]))] \neq 2$. Then we can and will ensure that the enclosure of the retraction of the monodromy at the end, namely $[\rho(\Psi([t_N]))]$, is contained in one of these four half-planes according to the following rules:
\begin{enumerate}
    \item if $[\mathrm{tr}(\Psi([t_N]))] < 2$, then $\rho(\Psi(T_0))\neq 1$ by Lemma~\ref{lemma:fibers}, so we choose the final time step sufficiently small such that $[\rho(\Psi([t_N]))]$ is contained in $U_2$, $U_3$ or $U_4$;
    \item if $[\mathrm{tr}(\Psi([t_N]))] > 2$, then $\rho(\Psi(T_0))\neq -1$ by Lemma~\ref{lemma:fibers}, in which case we choose the final time step such that $[\rho(\Psi([t_N]))]$ is contained in $U_1$, $U_3$ or $U_4$.
\end{enumerate}
Now we lift the path $t \mapsto \rho(\Psi(t))$ to a path $\theta:[0,T_0]\to\R$ in the universal cover, obtaining again enclosures $[\theta([t_j])]$. For lifting, we have used the coverings $U_1,\ldots, U_4$, which gives for each $j$ the inclusion
$$[\theta([t_j])] \subset \left(\frac{k_j-1}{2}\pi, \ \frac{k_j+1}{2}\pi\right), \quad k_j\in\Z$$ 
which is sufficient for our purposes. It is also possible to use the angle change of the enclosures $[\rho(\Psi([t_j]))]$ directly.
In any case, our choice of $U_i$ for the final enclosure $[\rho(\Psi([t_N]))]$ will ensure the following:
\begin{enumerate}
    \item if $[\mathrm{tr}(\Psi([t_N]))] < 2$, then $[\theta([t_N])]$ will lie in an open interval $( (k-1)\pi,(k+1)\pi)$ with $k$ an odd integer;
    \item if $[\mathrm{tr}(\Psi([t_N]))] > 2$, then $[\theta([t_N])]$ will lie in an open interval $( (k-1)\pi,(k+1)\pi)$ with $k$ an even integer.
\end{enumerate}
In both cases, it follows from Proposition~\ref{prop:index_computation} that $\mu_{CZ}(\psi)=k$.

{
\begin{remark}
    We note that this method to compute the Conley-Zehnder index can be applied to any starshaped level-set of a Hamiltonian on $\R^4$ with the standard symplectic structure. 
\end{remark}

\begin{proof}[Proof of Theorem~\ref{thm:index}]
    We subdivide the interval $[1.57617,1.57626]$ into small overlapping intervals. For each interval, we first compute non-rigorously an approximate value $z_0$ for the initial point $(z_0,0;0,w_2^-(z_0))$ of the periodic orbit $\gamma_c$. By considering the Poincar\'e map $h:C_+\to S_-$ and its derivative, we validate that on a neighborhood $I\subset \mathbb{R}$ of $z_0$ there is a unique zero of $s(z) = (pr_{w_2}\circ h)(z,0;0,w_2^-(z))$, establishing the existence of the family $\gamma_c$. Finally, we validate non-degeneracy of $\gamma_c$ using the linearized flow by the criterion of Lemma~\ref{lemma:transverse_nondegeneracy}.

    We repeat this process for fixed energies $c_1=1.57626$ and $c_2=1.57617$ to validate non-degeneracy of the double cover $\gamma_c^2$. We apply the above procedure for computing the Conley-Zehnder index to show that $\mu_{CZ}(\gamma_{c_1})=3$, $\mu_{CZ}(\gamma_{c_1}^2)=5$ and $\mu_{CZ}(\gamma_{c_2}^2)=6$. By non-degeneracy, it follows that $\mu_{CZ}(\gamma_c)=3$ for all $c\in[1.57617,1.57626]$.
\end{proof}

\begin{remark}
\label{rem:bifurcation_floer}
We briefly sketch how local Floer homology can be used to deduce the existence of a bifurcating orbit by applying it to the family $\gamma_c$ from Theorem 3.8.
Local Floer homology\footnote{Floer defined his homology theory, now known as Floer homology, \cite{Floer:fixed, salamon_zehnder_1992}, for non-degenerate (in the strong sense) periodic Hamiltonian orbits of time-dependent Hamiltonians. An extension to transversely non-degenerate periodic Hamiltonian orbits was described in \cite{Cieliebak_action}.}, described in detail in Section 3.2 of~\cite{Ginzburg_Conley_conjecture}, associates with a Hamiltonian periodic orbit a graded vector space, where the grading is given by the Conley-Zehnder index.
The key property we consider is its invariance for a uniformly isolated family of periodic orbits (see property LF1 in~\cite{Ginzburg_Conley_conjecture}).

We argue by contradiction and assume there is no orbit bifurcating out of the double cover $\gamma_c^2$.
This assumption implies that $\gamma_c^2$ is uniformly isolated for $c\in[1.57617, 1.57626]$. By the invariance property, the local Floer homology is independent of $c$ in this interval. 
However, from the index computations it follows (by Proposition~2.2 of~\cite{Cieliebak_action} and the index computations in section 2, 3 of \cite{Cieliebak_action}) that
\[
HF^{loc}_*(\gamma^2_{c=1.57626}) =
\begin{cases}
    \Z/2\Z & *=5,6 \\
    0 & \text{otherwise,}
\end{cases}
\quad
HF^{loc}_*(\gamma^2_{c=1.57617}) =
\begin{cases}
    \Z/2\Z & *=6,7 \\
    0 & \text{otherwise.}
\end{cases}
\]
Since these graded homology vector spaces are not isomorphic, we obtain a contradiction and conclude that there is at least one  orbit bifurcating out of $\gamma_c^2$.

Note that this argument proves the existence of a bifurcation with $C^1$-computations, but without the use of $C^r$-computations. We also remark that this type of argument applies specifically to Hamiltonian dynamical systems and provides less explicit information about the bifurcating orbit than the methods such as those given in~\cite{Wilczak_period_doubling}.
\end{remark}
}
\subsection{Proof of unknottedness and self-linking number}
We start with a simple lemma whose assumptions we can verify computationally.
\begin{lemma}
\label{lemma:box_isotopy}
Suppose there is a box $B \subset \{ (\mu,c)\in \R^2\}$, and intervals $Z, T \subset \R$ such that for all $(\mu,c)\in B$, the following holds.
\begin{itemize}
    \item There are $z \in Z$, $T_{\mu, c} \in T$ and an orbit $\gamma_{\mu,c}$ of $X_{K_{\mu,c}}$ such $\gamma_{\mu,c}$ is a non-degenerate symmetric $T_{\mu, c}$-periodic orbit with $\gamma_{\mu,c}(0)=(z,0;0,w^{\pm}_2(z))$.
    \item Furthermore, every symmetric periodic orbit passing through $(z',0;0,w^{\pm}_2(z'))$ with $z'\in Z$ of period $T'\in T$ is a reparametrization of $\gamma_{\mu,c}$.
\end{itemize}
Then, given a $C^r$-curve $s\mapsto (\mu_s,c_s)$, the map $s\mapsto T_{\mu_s, c_s}$ is of class $C^r$ and the map $s\mapsto \gamma_{\mu_s,c_s}(T_{\mu_s, c_s} \cdot \, )$ is a $C^r$-isotopy of circle embeddings.
\end{lemma}
The strong local uniqueness assumptions guarantee that the maps are well-defined. The smoothness claim follows from Theorem~2.2 of \cite{moser_zehnder_notes_2005} and its proof, which involves the implicit function theorem.

We now come to the claim of interest.
\begin{proposition}
\label{prop:self_linking}
Let $\gamma^r_{\mu,c}$ and $\gamma^d_{\mu,c}$ denote the retrograde and direct orbits of Theorem~\ref{thm:action_small}. Then these orbits are unknotted and have self-linking number $-1$.    
\end{proposition}

\begin{proof}
We take the horizontal curve $s\mapsto(\mu_s=s, c_s=c_0)$ for $s\in [0,1/2]$ with $c_0\in[2.1, 2.1+10^{-6}]$. Take the cover by boxes used in the proof of Theorem~\ref{thm:action_small}. Note that these boxes together with either the retrograde or the direct orbit family satisfy the assumptions of Lemma~\ref{lemma:box_isotopy}, giving us an isotopy on each box. Since we have checked uniqueness of the retrograde and direct orbits up to an action bound, the local isotopies patch together to isotopies $s \mapsto \gamma^r_{\mu_s,c_s}$ and $s \mapsto \gamma^d_{\mu_s,c_s}$ defined globally for $s\in [0,1/2]$. Thus, the retrograde and direct orbit families give rise to an isotopy of knots, all of which are reparametrizations of Reeb orbits. This means that unknottedness and the self-linking number are unchanged during isotopy, see Remarks~3.5.29 from \cite{geiges_introduction_2008}.

For $s=0$ we will check that the retrograde and direct orbit are unknotted and have self-linking number $-1$ by combining arguments from the literature.
McGehee has constructed a disk-like global surface of section in \cite{mcgehee_phd_thesis_1969} for $\mu=0$ with the retrograde orbit as binding orbit. We see that the retrograde orbit is unknotted, because we have a disk as Seifert surface.
{Furthermore, Obstruction~3 to the existence of disk-like global surfaces in Chapter~9 of \cite{frauenfelder_restricted_2018} asserts that the binding $\gamma$ of a disk-like surface of section to a Reeb flow on $S^3$ always has self-linking number $s\ell(\gamma)=-1$: this follows from the transversality of the flow to the interior and that fact that the Reeb flow is tangent to the boundary for a global surface of section.}

For $\mu=0$, the roles of the retrograde and direct orbit can be reversed as pointed out on page 29 of \cite{mcgehee_phd_thesis_1969}. This shows that the direct orbit is unknotted and has self-linking number $-1$. By the reasoning above, we conclude that the entire family of retrograde orbits and direct orbits (as a function of $\mu$) is unknotted and has self-linking number $-1$.
\end{proof}

\begin{proof}[Proof of Theorem~\ref{thm:small_action_unknotted_and_sl}]
Statements (1) and (2) follow from Theorem~\ref{thm:action_small} and Proposition~\ref{prop:self_linking}, respectively. For (3), we note that the proof of Proposition~\ref{prop:self_linking} gives non-degenerate families of periodic orbits, namely one family of retrograde orbits and another family of direct orbits. The Conley-Zehnder index is constant for each family. The value of this constant can be obtained by using the method of Section~\ref{sec:validated_index_computations}, giving us 3 and 5 for the retrograde and direct orbit, respectively.
Statement (4) can be proved similarly as non-degeneracy in the Moser regularization follows immediately from non-degeneracy in the Levi-Civita regularization, giving index 1 and 3 for the retrograde and direct orbit, respectively. 
\end{proof}

\section{Appendix}
\label{sec:appendix}
In this section we explain how to verify convexity of the compact component $\Sigma_{\mu,c}$ of $K_{\mu,c}^{-1}(0)$.
We start by observing that the Levi-Civita Hamiltonian can be rewritten as
\begin{equation}
    \label{eq:quad_form_LC_ham }
    K_{\mu,c}
=\frac{1}{2}|w+A(z)z|^2+ z^tB(z) z 
 -\frac{1-\mu}{2},
\end{equation}
where
$$
A(z)=
\left(
\begin{array}{cc}
0 & -(2|z|^2+\mu)\\
(2|z|^2-\mu) & 0
\end{array}
\right)
$$
and 
$$
B(z)=
\left(
\begin{array}{cc}
c-\frac{1}{2}(2|z|^2-\mu)^2 -\frac{\mu}{|2z^2-1|} & 0 \\
0 & c-\frac{1}{2}(2|z|^2+\mu)^2 -\frac{\mu}{|2z^2-1|}
\end{array}
\right).
$$

We need some crude a priori bounds on $|z|$ and $|w|$ for $(z,w)\in\Sigma_{\mu,c}$. Assume that $-c <H(L_1)$. To get a bound on $|z|$, we will work with the unregularized problem and bound $|q+\mu|$.
We recall from \cite{albers_contact_2012} that the maximal distance from $(-\mu,0)$ to the boundary of the Hill's region is attained on the $x$-axis at a point $(x_{\mu,c},0)$, where $x_{\mu,c} \in (-\mu,1-\mu)$, so $d_{\mu,c}:=x_{\mu,c}+\mu>0$.
\begin{lemma}
For $c\geq 2.1$ and $\mu \in [0,1/2]$, we have $d_{\mu,c}\leq \frac{3-2\mu}{5}$.
In particular $2|z|^2\leq 3/5$.
\end{lemma}
\begin{proof}
From the above observations on the Hill's region, it suffices to find a point on the axis between the two primaries that lies outside the Hill's region for $c\geq 2.1$.
Restricting the effective potential to the line segment between the primaries, we can write the negative effective potential as
$$
\bar U = \frac{1}{2}r^2 +\frac{1-\mu}{r} +\frac{\mu}{1-r} -\mu r +\frac{1}{2}\mu^2
$$
with $r\in [0,1]$. Plugging in $r=\frac{3-2\mu}{5}$, and bounding the maximum of this function of $\mu$, we see that maximum is attained at $\mu=1/2$, which equals $1253/600<2.1$, so the point $(\frac{3-2\mu}{5},0)$ lies outside the Hill's region for $c\geq 2.1$ and $\mu \in [0,1/2]$.
\end{proof}

\begin{lemma}
For $c\geq 2.1$ and $\mu \in [0,1/2]$, we have $|w| \leq 2$.
\end{lemma}
\begin{proof}
Note that the previous lemma implies that $|2|z|^2\pm \mu| \leq 3/5+1/2$ and $|2z^2-1|\geq 2/5$, so $B(z)$ is positive-definite for $(z,w)\in \Sigma_{\mu,c}$.
Hence \eqref{eq:quad_form_LC_ham } tells us that $|w+A(z)z|\leq \sqrt{1-\mu}$, so
$$
|w| \leq \sqrt{1-\mu} +|A(z)z|\leq 2,
$$
where we have again used that $2|z|^2\leq 3/5$.
\end{proof}

We know that $\Sigma_{\mu,c}$ is diffeomorphic to $S^3$ and we can trivialize $T \Sigma_{\mu,c}$ using the frame $v_1=I\nabla H, v_2=J \nabla H, v_3=K \nabla H$ from Section~\ref{sec:frames_invariants}. Define the tangential Hessian as the restriction of the Hessian of $K_{\mu,c}$ to $T\Sigma_{\mu,c}$. With respect to the quaternionic frame, we have the following matrix representation.
$$
THess_{ij}:= v_i^t Hess(K_{\mu,c}) v_j.
$$
\begin{lemma}
At the collision set the tangential Hessian is positive definite for all $c\geq -H(L_1)$.
\end{lemma}
\begin{proof}
The collision set corresponds to $z=0$. On this set the matrix representation of the tangential Hessian reduces to
\[
\left(
\begin{array}{ccc}
 \kappa|w|^2 &\mu \kappa  (w_2^2-w_1^2) &0\\ 
\mu \kappa (w_2^2-w_1^2)& 2 |w|^2
 \left( 1/2-{\mu}^{3}+ \left( c+1 \right) {\mu}^{2} \right) &2\kappa\mu w_1w_2 
\\
0&2\kappa\mu w_1w_2 & \kappa |w|^2
\end{array}
\right)
\]
where $\kappa = (2c-{\mu}^{2}-2\mu )$.
We note that the determinant of this matrix is
$$
|w|^6 \kappa^2 (\mu^2+1)^2 > 0.
$$
The collision locus is a circle, so connected. Hence it suffices to check positive-definiteness at a single point. To check positive-definiteness in a single point, set $w_2=0$. Then the Hessian reduces to
\[
w_1^2
\left(
\begin{array}{ccc}
 \kappa & -\mu \kappa  &0\\ 
-\mu \kappa & 2 \left( 1/2-{\mu}^{3}+ \left( c+1 \right) {\mu}^{2} \right) &0
\\
0&0 & \kappa 
\end{array}
\right)
.
\]
The upper $2\times 2$-block has determinant $\kappa (\mu^2+1)^2>0$ and its trace is positive, so all eigenvalues are positive. 
This proves the lemma.
\end{proof}

\begin{proposition}
    The Gauss-Kronecker curvature of $\Sigma_{\mu,c}$ is positive for $\mu\in [0,1/2]$ and $c\in [2.1,2.1+10^{-6}]$. Furthermore, the tangential Hessian is positive definite.
\end{proposition}

\begin{proof}
We cover the parameter domain $[0,1/2]\times [2.1,2.1+10^{-6}]$ by boxes of the form $[\mu_i,\mu_{i+1}] \times [2.1,2.1+10^{-6}]$. For each of the parameter boxes, we produce a covering of $\Sigma_{\mu,c}$ and verify the curvature condition.

Because of symmetry, it suffices to verify this curvature condition for $z_1\geq 0$. Also, by the above lemmas, we know that $\Sigma_{\mu,c}$ is contained in the hypercube $[-6/10,6/10] \times [-6/10, 6/10] \times [-2,2] \times [-2,2]$. We then proceed as follows:
\begin{itemize}
    \item We cover the hypercube $[0,6/10] \times [-6/10, 6/10] \times [-2,2] \times [-2,2]$ by boxes. For $\mu<0.1$ we can choose boxes with width $1/200$. For larger $\mu$ we need to shrink the width to $1/1000$. 
    \item For each box, check whether $2|z|^2> \frac{3-2\mu}{5}$. If so, then discard this box.
    \item For each box $B$, compute whether $K_{\mu,c}(B) <0$ or $K_{\mu,c}(B) >0$ using enclosures. Discard boxes for which this is the case.
    \item For each remaining box $B$, we have $0 \in K_{\mu,c}(B)$, and we have obtained a covering of $\Sigma_{\mu,c}$ (by mirroring in $z_1$).
    \item Form the tangential Hessian $THess$ as above and compute an enclosure of its determinant using interval arithmetic. 
\end{itemize}
The computer-assisted argument shows that $\det THess >0$ for all $(z,w)\in \Sigma_{\mu,c}$. This proves the first statement. Now observe that $\Sigma_{\mu,c}$ is connected. By the previous lemma $THess$ is positive definite at the collision locus. Since $\det THess>0$, there are no sign changes in any of the eigenvalues, so it follows that $THess$ is positive definite everywhere. This proves the proposition.    
\end{proof}

\begin{proof}[Proof of Theorem~\ref{thm:gss}]
    By the previous proposition and Lemma~12.1.2 of \cite{frauenfelder_restricted_2018} we know that the compact component of $\Sigma_{\mu,c}$ is convex. 
    It is known that convexity implies dynamical convexity, see for example \cite{hofer_wysocki_zehnder_1998} or Theorem~12.2.1 in \cite{frauenfelder_restricted_2018}.
    Together with Theorem~\ref{thm:small_action_unknotted_and_sl} and Hryniewicz's theorem~\ref{thm:Hryniewicz_gss}, the conclusion follows.
\end{proof}

\section*{Acknowledgements}
We are grateful to the referees for their helpful suggestions. Chankyu Joung and Otto van Koert were supported by the National Research Foundation of Korea (NRF), grant number (MSIT) (2023R1A2C1005562), funded by the Korea government.

\printbibliography

\end{document}